\documentclass[10pt]{amsart}
\usepackage{amsmath,amssymb,amsthm}
\usepackage[margin=1.1in]{geometry}
\usepackage{microtype}
\usepackage{xcolor}
\usepackage[colorlinks=true,linkcolor=blue!50!black,citecolor=blue!50!black,urlcolor=blue!50!black]{hyperref}

\newtheorem{theorem}{Theorem}

\newtheorem{proposition}[theorem]{Proposition}

\newcommand{\C}{\mathbb C}

\title[Nineteen vectors for phase retrieval in $\C^6$]{Nineteen vectors and their phase-retrieval injectivity in $\C^6$}

\author[Minwook Kim]{Minwook Kim\\[3pt]
{\small KAIST, Daejeon, South Korea}}
\address{KAIST (Korea Advanced Institute of Science and Technology), Daejeon, South Korea}
\email{brian10kim22@kaist.ac.kr}

\subjclass[2020]{42C15, 94A12}

\date{August 2026}

\begin{document}

\begin{abstract}
We exhibit nineteen explicit vectors in $\C^6$, with Gaussian-integer
entries, that do phase retrieval: the intensities
$|\langle a_j,x\rangle|^2$, $j=1,\dots,19$, determine every $x\in\C^6$
up to a unimodular scalar. This gives $m_{\C}(6)\le19=4d-5$, one below
the generic count $4d-4$; combined with the lower bound
$m_{\C}(6)\ge18$ of Wang and Xu, the minimal measurement number in
dimension six is $18$ or $19$. The claim is verified by an exact
rational computation, archived and rerunnable.
\end{abstract}

\maketitle

\section{Statement}

Fix measurement vectors $a_1,\dots,a_m\in\C^d$ and let
\[
\Phi_A(x)=\bigl(|\langle a_1,x\rangle|^2,\dots,|\langle
a_m,x\rangle|^2\bigr),\qquad
\langle a,x\rangle=\textstyle\sum_k\bar a_kx_k .
\]
The family \emph{does phase retrieval} if $\Phi_A(x)=\Phi_A(y)$ forces
$y=\omega x$ with $|\omega|=1$; $m_\C(d)$ denotes the minimal $m$ for
which such a family exists in $\C^d$. Generic families of $4d-4$
vectors do phase retrieval \cite{CEHV}, and Vinzant \cite{Vinzant}
exhibited $11=4\cdot4-5$ vectors doing phase retrieval in $\C^4$. In
dimension six, Wang and Xu proved $m_\C(6)\ge18$
\cite[Theorem~5.2]{WangXu}, while the best upper bound known has been
the generic $4d-4=20$; whether $19=4d-5$ vectors can do phase
retrieval in $\C^6$ has been open. We answer affirmatively.

\begin{theorem}\label{thm:main}
The nineteen row vectors of Table~\textup{\ref{tab:frame}} do phase
retrieval in $\C^6$. Consequently $18\le m_\C(6)\le19$.
\end{theorem}

\begin{table}[h]
{\scriptsize
\begin{verbatim}
A6 = [
 [1, -5, 25, -125, 625, -3125],
 [1, -4, 16, -64, 256, -1024],
 [1, -3, 9, -27, 81, -243],
 [1, -2, 4, -8, 16, -32],
 [1, -1, 1, -1, 1, -1],
 [1, 0, 0, 0, 0, 0],
 [1, 1, 1, 1, 1, 1],
 [1, 2, 4, 8, 16, 32],
 [1, 3, 9, 27, 81, 243],
 [1, 4, 16, 64, 256, 1024],
 [1, 5, 25, 125, 625, 3125],
 [9765625, 13671875+46875000j, -205859375+131250000j,
  -918203125-804375000j, 2575515625-5533500000j, 30166521875+4615575000j],
 [9765625, 10937500+37500000j, -131750000+84000000j,
  -470120000-411840000j, 1054931200-2266521600j, 9884965888+1512431616j],
 [9765625, 8203125+28125000j, -74109375+47250000j,
  -198331875-173745000j, 333786825-717141600j, 2345748741+358907112j],
 [9765625, 5468750+18750000j, -32937500+21000000j,
  -58765000-51480000j, 65933200-141657600j, 308905184+47263488j],
 [9765625, 2734375+9375000j, -8234375+5250000j,
  -7345625-6435000j, 4120825-8853600j, 9653287+1476984j],
 [9765625, -2734375-9375000j, -8234375+5250000j,
  7345625+6435000j, 4120825-8853600j, -9653287-1476984j],
 [9765625, -8203125-28125000j, -74109375+47250000j,
  198331875+173745000j, 333786825-717141600j, -2345748741-358907112j],
 [9765625, -10937500-37500000j, -131750000+84000000j,
  470120000+411840000j, 1054931200-2266521600j, -9884965888-1512431616j],
]
\end{verbatim}}
\caption{The nineteen measurement vectors $a_1,\dots,a_{19}$ for
$\C^6$ (rows; $\mathtt j=\sqrt{-1}$).}\label{tab:frame}
\end{table}

\section{Verification}\label{sec:verification}

The verification is a finite, exact computation; this section records
the elementary reduction that produces the computed system, and the
archived facts about that system.

\subsection*{Reduction to polynomials}
Every row of Table~\ref{tab:frame} is a scalar multiple of a vector
$a(w):=(1,\bar w,\bar w^2,\dots,\bar w^5)$: rows $1$--$11$ are $a(r)$
with $r=-5,\dots,5$, and rows $12$--$19$ are $25^5\,a(\zeta s)$ with
\[
\zeta=\tfrac{-7+24i}{25},\qquad s\in\{-5,-4,-3,-2,-1,1,3,4\},
\]
as one checks by expanding the powers. Scaling a measurement vector by
a nonzero constant scales the corresponding intensity by a positive
constant, so it suffices to verify phase retrieval for the unscaled
family $\{a(r)\}\cup\{a(\zeta s)\}$. For $x\in\C^6$ let
$f_x(z)=\sum_{k=0}^{5}x_{k+1}z^k$, so that $\langle
a(w),x\rangle=f_x(w)$; the map $x\mapsto f_x$ is a linear bijection
onto the polynomials of degree $\le5$. For $f(z)=\sum c_kz^k$ write
$f^{\#}(z)=\sum\bar c_kz^k$; then $f^{\#}(r)=\overline{f(r)}$ for real
$r$, $|f(z)|^2=f(z)f^{\#}(\bar z)$ in general, and $(f^{\#})^{\#}=f$.

For real polynomials $U,V$ put
\[
\mathcal W(U,V)(t)
=\tfrac1{2i}\Bigl(U(\zeta t)V(\bar\zeta t)-U(\bar\zeta
t)V(\zeta t)\Bigr).
\]
Its values at real $t$ are real, so its coefficients are real; if
$\deg U,\deg V\le k$, the coefficient of $t^m$ collects the terms
$u_iv_j-u_jv_i$ with $i<j$, $i+j=m$, so $\deg\mathcal W(U,V)\le2k-1$
and $t\mid\mathcal W(U,V)$; the coefficient of $t^1$ is
$\sin\theta\,(u_1v_0-u_0v_1)$ with $\theta=\arg\zeta$,
$\sin\theta=\tfrac{24}{25}$. A direct expansion with $u=U+iV$,
$u^{\#}=U-iV$ gives, for real $s$,
\begin{equation}\label{eq:fourdelta}
|u(\zeta s)|^2-|u^{\#}(\zeta s)|^2
=u(\zeta s)u^{\#}(\bar\zeta s)-u(\bar\zeta s)u^{\#}(\zeta s)
=4\,\mathcal W(U,V)(s).
\end{equation}

Two elementary facts. \emph{First}: if $f,g\neq0$ have degree $\le5$
and $|f|=|g|$ at the eleven points $r=-5,\dots,5$, then
$ff^{\#}-gg^{\#}$, of degree $\le10$, vanishes at eleven points, so
$ff^{\#}=gg^{\#}$; writing $h=\gcd(f,g)$, $f=hu$, $g=hv$ with
$\gcd(u,v)=1$ and cancelling $hh^{\#}$ gives $uu^{\#}=vv^{\#}$, whence
$v$ divides $u^{\#}$ and $u$ divides $v^{\#}$, so by degrees
$v=\lambda u^{\#}$ with $|\lambda|=1$: every such pair has the form
\begin{equation}\label{eq:normalform}
f=hu,\qquad g=\omega\,h\,u^{\#},\qquad|\omega|=1,\qquad \deg h+\deg
u\le5 .
\end{equation}
\emph{Second}: if real $U,V$ satisfy $\mathcal W(U,V)\equiv0$, they
are linearly dependent. Indeed $\zeta^2$ is not a root of unity ---
otherwise $\zeta$ would be one, making
$\zeta+\zeta^{-1}=2\operatorname{Re}\zeta=-\tfrac{14}{25}$ a rational
algebraic integer, hence an integer --- and for coprime $U,V$ the
identity $U(\zeta t)V(\bar\zeta t)=U(\bar\zeta t)V(\zeta t)$ forces
$U(\zeta t)\mid U(\bar\zeta t)$ (a common root $\tau$ of $U(\zeta t),
V(\zeta t)$ would make $\zeta\tau$ a common root of $U,V$), so
$U(\bar\zeta t)=\lambda U(\zeta t)$ and the root multiset of $U$ is
invariant under multiplication by $\zeta^2$, which has infinite order:
$U=u\,t^p$, symmetrically $V=v\,t^q$, $\min(p,q)=0$ by coprimality,
and substituting back forces $\zeta^{2(p-q)}=1$, so $p=q=0$ and both
are constants.

\begin{proposition}\label{prop:reduction}
The unscaled family fails phase retrieval if and only if there exist
real linearly independent $U,V$ of degree $\le5$ and a real
$\gamma\neq0$ with
\begin{equation}\label{eq:standard}
\mathcal W(U,V)(t)\;=\;\gamma\;t\,(t+5)(t+4)(t+3)(t+2)(t+1)(t-1)(t-3)(t-4).
\end{equation}
Moreover every such pair satisfies $u_1v_0-u_0v_1\neq0$.
\end{proposition}

\begin{proof}
($\Leftarrow$) Set $u=U+iV$ and let $x,y$ be the coefficient vectors
of $u,u^{\#}$. Then $|u(r)|=|u^{\#}(r)|$ for all real $r$, and at the
eight rotated points \eqref{eq:fourdelta} and \eqref{eq:standard} give
$|u(\zeta s)|=|u^{\#}(\zeta s)|$; so all nineteen measurements agree.
If $u^{\#}=\omega u$ with $|\omega|=1$, then $(1-\omega)U=i(1+\omega)V$
would make $U,V$ dependent; hence $y$ is not a unimodular multiple of
$x$, and phase retrieval fails.

($\Rightarrow$) Let $f=f_x$, $g=f_y$ witness a failure. If $f=0$, then
$g$ vanishes at eleven $>5$ real points, so $g=0=f$, a contradiction;
so $f,g\neq0$ and \eqref{eq:normalform} applies, with
$u\not\parallel u^{\#}$ (else $g$ is a unimodular multiple of $f$),
i.e.\ $U,V$ independent for $u=U+iV$. Let $k=\deg u$, so $\deg
h\le5-k$. At each rotated point with $h(\zeta s)\neq0$, dividing
$|f(\zeta s)|=|g(\zeta s)|$ by $|h(\zeta s)|$ and using
\eqref{eq:fourdelta} gives $\mathcal W(U,V)(s)=0$; since $h$ has at
most $5-k$ roots, at least $8-(5-k)=k+3$ of the eight values $s$ are
roots of $\mathcal W:=\mathcal W(U,V)$. Write $\mathcal
W=t\widetilde{\mathcal W}$, $\deg\widetilde{\mathcal W}\le2k-2$. If
$k\le4$, then $k+3>2k-2$: $\widetilde{\mathcal W}$ has more distinct
nonzero roots than its degree, so $\mathcal W\equiv0$ and the second
fact above makes $U,V$ dependent --- a contradiction. Hence $k=5$, $h$
is constant, all eight equations survive, and $\widetilde{\mathcal
W}$, of degree $\le8$, has the eight listed roots:
$\widetilde{\mathcal W}=\gamma\prod(t-s)$ with $\gamma\neq0$ (else
$\mathcal W\equiv0$ again), which is \eqref{eq:standard}. Comparing
coefficients of $t^1$ in \eqref{eq:standard} gives
$\tfrac{24}{25}(u_1v_0-u_0v_1)=\gamma\prod_s(-s)\neq0$.
\end{proof}

\subsection*{The computation}
Equation \eqref{eq:standard} is linear in the Pl\"ucker coordinates
$q_{ij}=u_jv_i-u_iv_j$ of the plane spanned by $U,V$. In the chart
$q_{01}=u_1v_0-u_0v_1=1$, legitimate by the last clause of
Proposition~\ref{prop:reduction}, the coefficient equations of
\eqref{eq:standard} reduce to five quadratic equations in five
unknowns with rational coefficients --- written out in full, with
their derivation, in Appendix~\ref{app:system}. The archived
computation \cite{Artifacts} certifies, in exact rational arithmetic
with no floating point:
\begin{itemize}
\item the system is zero-dimensional, and its eliminant with respect
to one coordinate is a squarefree integer polynomial of degree
fourteen (matching the B\'ezout count for a proper linear section of
the Pl\"ucker-embedded $\mathrm{Gr}(2,6)$, whose degree is
$C_4=14$);
\item the Sturm sequence of the eliminant has seven sign variations
at $-\infty$ and seven at $+\infty$: the eliminant has \emph{no real
roots}. A real solution of the system would make its eliminated
coordinate a real root; hence no real solution exists, no pair
$(U,V,\gamma)$ as in Proposition~\ref{prop:reduction} exists, and
Theorem~\ref{thm:main} follows.
\end{itemize}
Protocol: the Gr\"obner-basis computation was performed with
\texttt{msolve} \cite{msolve} over $\mathbb Q$; the real-root count is
independently re-derived from the archived eliminant by a
Sturm-sequence script in integer arithmetic; the polynomial system was
generated by two independently written emitters and compared as
canonically normalized sets before solving; as a sanity check, the
$19\times36$ real matrix of the lifted measurements
$\bar a_ja_j^{\top}$ has rank $19$ over $\mathbb Q$; and the same
pipeline, run on configurations engineered to fail, reports the
expected nonzero real-solution counts. The \texttt{msolve} input, the
certificate data (eliminant, Sturm sequence, sign counts, SHA-256
hashes), and a standalone exact verifier are archived at
\cite{Artifacts} (directory \texttt{certificate/}); the verification
reruns in minutes on a laptop.

\subsection*{Acknowledgment}
The author received assistance from an AI assistant in the search for
candidate frames and in the preparation of this manuscript.

\appendix

\section{The certified system}\label{app:system}

This appendix specifies the certified system exactly as archived
(\texttt{certificate/certificate.json}); the verifier
(\texttt{certificate/verify\_exact.py}) rebuilds it from the
description below and rechecks every claim.

Write $U=\sum_{i=0}^{5}u_it^i$, $V=\sum_{i=0}^{5}v_it^i$ and
$q_{ij}=u_jv_i-u_iv_j$. Since
$(\zeta^i\bar\zeta^{\,j}-\bar\zeta^{\,i}\zeta^j)/2i
=\sin\bigl((i-j)\theta\bigr)$ with $\theta=\arg\zeta$, the
coefficients of $\mathcal W$ are
\[
[t^m]\,\mathcal W(U,V)=\sum_{i<j,\ i+j=m}s_{j-i}\,q_{ij},
\qquad
s_k:=\sin k\theta=\operatorname{Im}\zeta^k,
\]
and exactly
\[
s_1=\tfrac{24}{25},\quad
s_2=-\tfrac{336}{625},\quad
s_3=-\tfrac{10296}{15625},\quad
s_4=\tfrac{354144}{390625},\quad
s_5=\tfrac{1476984}{9765625}.
\]
Let
\[
g(t)=\prod_{s}(t-s)
=t^8+7t^7-16t^6-182t^5-91t^4+1183t^3+1546t^2-1008t-1440,
\]
the product over the eight listed $s$, with coefficients
$g_0,\dots,g_8$ in ascending order.

\emph{Chart.} A witness $(U,V,\gamma)$ of
Proposition~\ref{prop:reduction} has $q_{01}\neq0$, and replacing
$(U,V)$ by a basis of the same plane with change-of-basis determinant
$q_{01}^{-1}$ scales every $q_{ij}$, hence $\mathcal W$ and $\gamma$,
by $q_{01}^{-1}$; so we may take $q_{01}=1$. Equation
\eqref{eq:standard} then says
\[
\sum_{i<j,\ i+j=m}s_{j-i}\,q_{ij}\;=\;\gamma\,g_{m-1},
\qquad m=1,\dots,9 .
\]
The $m=1$ equation gives $\gamma=s_1/g_0$ (so $\gamma\neq0$ is
automatic). In this chart the three-term Pl\"ucker relation on the
indices $(0,1,i,j)$ gives $q_{ij}=q_{0i}q_{1j}-q_{0j}q_{1i}$ for
$2\le i<j\le5$, so the eight coordinates $q_{0k},q_{1k}$
($k=2,\dots,5$) determine the plane. The equations $m=2,3,4$ are
linear and solve to
\[
q_{02}=\frac{s_1g_1}{s_2g_0}=-\frac54,\qquad
q_{12}=\frac{g_2}{g_0}-\frac{s_3}{s_1}\,q_{03},\qquad
q_{13}=\frac{s_1g_3/g_0-s_4\,q_{04}}{s_2}.
\]
Five unknowns remain; following the archived naming, put
\[
\beta_4=q_{03},\quad \beta_5=q_{04},\quad \beta_6=q_{05},\quad
\alpha_5=-q_{14},\quad \alpha_6=-q_{15}
\]
(and $\beta_3=q_{02}$, $\alpha_3=-q_{12}$, $\alpha_4=-q_{13}$ for the
quantities already determined above). Substituting into the equations
$m=5,\dots,9$ and clearing denominators to primitive integer form
yields the archived system, verbatim:
{\small
\begin{align*}
E_1:&\ \ 617760000\,\alpha_5-617760000\,\beta_4^2+966250000\,\beta_4
-1897200000\,\beta_5\\
&\ \ +141790464\,\beta_6-1707265625=0,\\
E_2:&\ \ -225000\,\alpha_5-303552\,\alpha_6+123552\,\beta_4\beta_5
-193250\,\beta_5-40625=0,\\
E_3:&\ \ -225000000\,\alpha_5\beta_4-193050000\,\alpha_6
+106007616\,\beta_4\beta_6-379440000\,\beta_5^2\\
&\ \ -330078125\,\beta_5-165808500\,\beta_6-2500000=0,\\
E_4:&\ \ 360000\,\alpha_6\beta_4+607104\,\beta_5\beta_6
+528125\,\beta_6+3125=0,\\
E_5:&\ \ 1440\,\alpha_5\beta_6-1440\,\alpha_6\beta_5+1=0.
\end{align*}}%
A real witness of Proposition~\ref{prop:reduction} produces, through
the (real) normalization above, a real solution of
$E_1=\dots=E_5=0$.

\emph{The eliminant.} Let
$I\subset\mathbb Q[\beta_4,\beta_5,\alpha_5,\alpha_6,\beta_6]$ be the
ideal generated by $E_1,\dots,E_5$. The archived eliminant
$E(\beta_5)$ is the primitive squarefree integer polynomial of degree
$14$ generating $I\cap\mathbb Q[\beta_5]$, found via a grevlex
Gr\"obner basis followed by FGLM conversion to lexicographic order ---
but no property of that computation is needed beyond the certified
fact $E\in I$, which the verifier checks by exact reduction of $E$ to
zero against the archived basis. Membership alone yields the logic of
the certificate: at any solution of the system, real or complex, the
$\beta_5$-coordinate is a root of $E$; the Sturm sequence of $E$ has
seven sign variations at $-\infty$ and seven at $+\infty$, so $E$ has
no real roots; hence the system has no real solutions. (In
particular, neither zero-dimensionality nor completeness of the
elimination is used in this direction; both hold and are certified,
but only as cross-checks.)

\end{document}